\documentclass[a4paper,11pt,reqno]{amsart}
\usepackage{graphicx}
\usepackage{amsmath}
\usepackage{amstext}
\usepackage{amsbsy}
\usepackage{amsopn}
\usepackage{upref}
\usepackage{amsthm}
\usepackage{amsfonts}
\usepackage{amssymb}
\usepackage{mathrsfs}
\usepackage{times}
\usepackage{cite}
\usepackage{bm}
\allowdisplaybreaks
 \newtheorem{theorem}{Theorem}
 
 \newtheorem{proposition}[theorem]{Proposition}
 
\theoremstyle{definition}
 \newtheorem{definition}[theorem]{Definition}
\theoremstyle{remark}
 
\begin{document}
\title[Bargmann transform]{Bargmann transform on the space of hyperplanes}
\author[H.~Chihara]{Hiroyuki Chihara}
\address{College of Education, University of the Ryukyus, Nishihara, Okinawa 903-0213, Japan}
\email{aji@rencho.me}
\thanks{Supported by the JSPS Grant-in-Aid for Scientific Research \#19K03569.}
\begin{abstract}
We introduce an integral transform on the space of hyperplanes applying the Plancherel formula of the Radon transform to the definition of the semiclassical Bargmann transform on the Euclidean space. This is similar to the semiclassical Bargmann transform and some basic facts on microlocal analysis are also discussed. 
\end{abstract}
\maketitle
\section{Introduction}
\label{section:introduction}
In the present paper we introduce semiclassical holomorphic Fourier integral operators on the space of hyperplanes in the $n$-dimensional Euclidean space $\mathbb{R}^n$ with $n \geqq 2$, and discuss some basic properties of them. These operators are related to the Bargmann transform and the Radon transform on $\mathbb{R}^n$. The purpose of the present paper is to observe the microlocal analysis of the Radon transform in view of analytic microlocal analysis introduced by Sj\"ostrand in \cite{Sjoestrand1}. Throughout the present paper we denote by $h>0$ the semiclassical parameter. The unit of imaginary number $\sqrt{-1}$ is denoted by $i$. For $z=(z_1,\dotsc,z_n)$ and $\zeta=(\zeta_1,\dotsc,\zeta_n)$ in $\mathbb{C}^n$, set
$$
z\zeta=\langle{z,\zeta}\rangle
:=
z_1\zeta_1+\dotsb+z_n\zeta_n,
\quad
z^2
=
\langle{z,z}\rangle, 
\quad
\lvert{z}\rvert^2
:=
\langle{z,\bar{z}}\rangle.
$$
The Bargmann transform of a tempered distribution $u$ on $\mathbb{R}^n$ is defined by 
$$
\mathcal{T}_hu(z)
:=
2^{n/4}
h^{-3n/4}
\int_{\mathbb{R}^n} 
e^{-\pi(z-y)^2/h}
u(y)
dy,
\quad
z\in\mathbb{C}^n. 
$$
We shall present the basic properties of $\mathcal{T}_h$ in Section~\ref{section:origin}.
\par
Following Helgason's celebrated textbook \cite{Helgason}, 
we begin by recalling the well-known facts on the Radon transform. 
After that we introduce our integral transform on the space of hyperplanes 
and state our results of the present paper.
Let us recall the parametrization of hyperplanes in $\mathbb{R}^n$ 
to define the Radon transform on $\mathbb{R}^n$. 
Denote by $\mathbb{S}^{n-1}$ 
the unit sphere in $\mathbb{R}^n$ with its center at the origin. 
For $(\omega,t) \in \mathbb{S}^{n-1}\times\mathbb{R}$, set 
$$
H(\omega,t)
:=
\{x\in\mathbb{R}^n\ : \ x\omega=t\}
=
\{x\in\mathbb{R}^n\ : \ (x-t\omega)\omega=0\}. 
$$
This is a hyperplane passing through $t\omega$ 
and perpendicular to $\omega$ in $\mathbb{R}^n$. 
The absolute value $\lvert{t}\rvert$ is the distance between $H(\omega,t)$ and the origin. 
Note that $H(-\omega,-t)=H(\omega,t)$. 
Denote the set of all hyperplanes in $\mathbb{R}^n$ by $\mathbb{P}^n$. 
Then $\mathbb{S}^{n-1}\times\mathbb{R}$ 
is a double covering of $\mathbb{P}^n$, 
and $\mathbb{P}^n$ 
is regarded as the set of all equivalence classes of 
$\mathbb{S}^{n-1}\times\mathbb{R}$ 
provided that $(\omega,t)$ and $(-\omega,-t)$ are identified. 
\par
Here we introduce the Radon transform on $\mathbb{R}^n$. 
Denote by $L^1_\text{loc}(\mathbb{R}^n)$ 
the set of all locally integrable functions on $\mathbb{R}^n$. 
For $f(x) \in L^1_\text{loc}(\mathbb{R}^n)$ 
satisfying 
$f(x)=\mathcal{O}(\lvert{x}\rvert^{-(n-1)-\delta})$ 
as $\lvert{x}\rvert \rightarrow \infty$ 
with some $\delta>0$, 
we define the Radon transform of $f$ by 
$$
\mathcal{R}f(\omega,t)
:=
\int_{H(\omega,t)}
f(x)
dm(x)
=
\int_{\omega^\perp}
f(t\omega+y)
dm(y),
\quad
(\omega,t)\in\mathbb{S}^{n-1}\times\mathbb{R}, 
$$
where $m$ is the Lebesgue measure on hyperplanes induced by 
the standard Euclidean metric on $\mathbb{R}^n$, 
and $\omega^\perp$ is the orthogonal complement of 
$\{\omega\}$ in $\mathbb{R}^n$. 
$\mathcal{R}f(\omega,t)$ is a function on $\mathbb{P}^n$, 
that is, $\mathcal{R}f(-\omega,-t)=\mathcal{R}f(\omega,t)$ holds. 
Similarly, we say that $U(\omega,t)$ is a function on 
$\widetilde{\mathbb{P}^n}$ if 
$U(\omega,t)$ is a function on $\mathbb{S}^{n-1}\times\mathbb{R}$ 
satisfying $U(-\omega,-t)=(-1)^{n-1}U(\omega,t)$, that is, 
$U(-\omega,-t)=U(\omega,t)$ for odd $n$ and 
$U(-\omega,-t)=-U(\omega,t)$ for even $n$. 
Let us denote by $\mathscr{E}^\prime(\Omega)$ 
the set of all compactly supported distributions on a space $\Omega$. 
The Radon transform $\mathcal{R}$ can be extended to any 
$u \in \mathscr{E}^\prime(\mathbb{R}^n)$ 
and $\mathcal{R}u$ belongs to $\mathscr{E}^\prime(\mathbb{P}^n)$. 
See \cite[Definition in page 28]{Helgason}. 
\par
We now introduce some function spaces. 
Let us denote by $\mathscr{S}(\mathbb{R}^n)$ 
the Fr\'echet space of all rapidly decreasing functions on $\mathbb{R}^n$, 
and let us denote by $\mathscr{S}^\prime(\mathbb{R}^n)$ 
its topological dual, that is, 
the set of all tempered distributions on $\mathbb{R}^n$. 
Moreover we denote by $\mathscr{S}(\mathbb{S}^{n-1}\times\mathbb{R})$ 
the set of all smooth functions on $\mathbb{S}^{n-1}\times\mathbb{R}$ 
satisfying 
$$
\sup_{\mathbb{S}^{n-1}\times\mathbb{R}}
\left\lvert
t^lP
\frac{\partial^kU}{\partial t^k}(\omega,t)
\right\rvert
<\infty
$$
for any $k,l=0,1,2,\dotsc$ and 
for any differential operator $P$ of order $0,1,2,\dotsc$ 
on $\mathbb{S}^{n-1}$ with smooth coefficients. 
These suprema form a family of continuous seminorms of 
$\mathscr{S}(\mathbb{S}^{n-1}\times\mathbb{R})$. 
Its topological dual is denoted by 
$\mathscr{S}^\prime(\mathbb{S}^{n-1}\times\mathbb{R})$. 
If we replace $\mathbb{S}^{n-1}\times\mathbb{R}$ 
by $\mathbb{P}^n$ or $\widetilde{\mathbb{P}^n}$, 
we can define 
$\mathscr{S}(\mathbb{P}^n)$ 
and 
$\mathscr{S}^\prime(\mathbb{P}^n)$, 
or 
$\mathscr{S}(\widetilde{\mathbb{P}^n})$ 
and 
$\mathscr{S}^\prime(\widetilde{\mathbb{P}^n})$ 
similarly. 
Throughout the present paper we only deal with distributions 
which are independent of the semiclassical parameter $h>0$. 
\par
Next we recall the inversion formula of the Radon transform. 
It is well-known that the range of the Radon transform of 
$\mathscr{S}(\mathbb{R}^n)$ is characterized as 
$\mathscr{S}_\text{H}(\mathbb{P}^n)$,  
which is the set of all $U(\omega,t) \in \mathscr{S}(\mathbb{P}^n)$ 
satisfying the moment condition that  
$$
\int_{-\infty}^\infty
U(\omega,t)
t^k
dt
$$
is a homogeneous polynomial of degree $k$ 
in $\omega\in\mathbb{S}^{n-1}$ for all $k=0,1,2,\dotsc$. 
Then the Radon transform $\mathcal{R}$ is a linear bijection of 
$\mathscr{S}(\mathbb{R}^n)$ onto $\mathscr{S}_\text{H}(\mathbb{P}^n)$, 
and the inversion formula is given by 
$$
u
=
C_n^{-1}(-\Delta_{\mathbb{R}^n})^{(n-1)/2}\mathcal{R}^\ast\mathcal{R}u
=
C_n^{-1}\mathcal{R}^\ast\bigl(\lvert{D_t}\rvert^{(n-1)/2}\mathcal{R}u\bigr)
$$
for $u\in\mathscr{S}(\mathbb{R}^n)$, 
where 
\begin{equation}
C_n
=
(4\pi)^{(n-1)/2}
\Gamma(n/2)\Gamma(1/2), 
\label{equation:constant}
\end{equation}
$\Gamma(\cdot)$ is the gamma function, 
$\mathcal{R}^\ast$ is the dual Radon transform defined by 
$$
\mathcal{R}^\ast U(x)
:=
\int_{\mathbb{S}^{n-1}}
U(\omega,x\omega)
d\omega
$$
for $U(\omega,t) \in C(\mathbb{P}^n)$, 
$d\omega$ is the standard Lebesgue measure on $\mathbb{S}^{n-1}$, 
and 
$$
\Delta_{\mathbb{R}^n}
:=
\sum_{j=1}^n
\frac{\partial^2}{\partial x_j^2},
\quad
\lvert{D_t}\rvert
:=
\left(-\frac{d^2}{dt^2}\right)^{1/2}. 
$$
\par
We now introduce our integral transform on 
$\widetilde{\mathbb{P}^n}$, 
which is the main object of the present paper. 
For $U \in \mathscr{S}^\prime(\widetilde{\mathbb{P}^n})$, 
set 
$$
\mathcal{B}_hU(z)
:=
A_nh^{-3n/4}
\iint_{\mathbb{P}^n}
e^{-\pi(z\omega-t)^2/h}
H_{n-1}\left(\sqrt{\frac{\pi}{h}}(z\omega-t)\right)
U(\omega,t)
d\omega
dt,
\quad
z\in\mathbb{C}^n,
$$
where $A_n$ is the constant defined by 
$$
A_{2k+1}
=\frac{(-1)^kk!}{2^{k/2-1/4}(2k)!},
\quad
A_{2k}
=
\frac{(-1)^{k-1/2}\pi^{1/2}}{2^{3k/2}(k-1)!},
\quad
k=1,2,3,\dotsc,
$$
$H_l(s)$ is the Hermite polynomial of $s\in\mathbb{C}$ of order 
$l=1,2,3,\dotsc$
defined by Rodrigues' formula or the explicit formula
$$
H_l(s)
=
(-1)^l
e^{s^2}
\frac{d^l}{ds^l}
e^{-s^2}
=
l!
\sum_{j=0}^{[l/2]}
\frac{(-1)^j}{j!(l-2j)!}
(2s)^{l-2j}, 
$$
and $[\cdot]$ is the floor function, that is, 
$$
[w]
:=
\max\{k\in\mathbb{Z}\ : \ k \leqq w\},
\quad w\in\mathbb{R}. 
$$
Note that 
$H_{n-1}(\sqrt{\pi}(z\omega-t)/\sqrt{h})$ 
is a function on $\widetilde{\mathbb{P}^n}$ 
and 
$H_{n-1}(\sqrt{\pi}(z\omega-t)/\sqrt{h})U(\omega,t)$ 
is in $\mathscr{S}^\prime(\mathbb{P}^n)$ 
for any fixed $z\in\mathbb{C}^n$. 
So the integration $\mathcal{B}_hU(z)$ is well-defined for any fixed $z\in\mathbb{C}^n$ 
since $e^{-\pi(z\omega-t)^2/h} \in \mathscr{S}(\mathbb{P}^n)$. 
Moreover $\mathcal{B}_hU(z)$ is holomorphic in $\mathbb{C}^n$. 
Set 
$\phi(z,\omega,t):=i\pi(z\omega-t)^2$ 
and 
$a_n(z,\omega,t):=H_{n-1}(\sqrt{\pi}(z\omega-t)/\sqrt{h})$ 
for short. 
Then $\mathcal{B}_h$ is given by 
$$
\mathcal{B}_hU(z)
=
A_nh^{-3n/4}
\iint_{\mathbb{P}^n}
e^{i\phi(z,\omega,t)/h}
a_n(z,\omega,t)
U(\omega,t)
d\omega
dt. 
$$
We make use of the theory of semiclassical Fourier integral operators 
with complex phase functions and analytic microlocal analysis 
originated by Sj\"ostrand in \cite{Sjoestrand1}. 
See also \cite{Sjoestrand2} and \cite{HitrikSjoestrand}. 
Denote by $\zeta$ local coordinates of $\mathbb{P}^n$. 
More precisely, for 
$$
(\omega,t)=(\omega_1,\dotsc,\omega_n,t)\in\mathbb{S}^{n-1}\times\mathbb{R}
$$  
if $\omega_1^2+\dotsb+\omega_{n-1}^2<1/2$, 
then $\omega_n=\pm\sqrt{1-\omega_1^2-\dotsb-\omega_{n-1}^2}$ and we choose 
$\zeta=(\omega_1,\dotsc.\omega_{n-1},t)$ as local coordinates. 
We often express the real and imaginary parts of $z\in\mathbb{C}^n$ 
by $x\in\mathbb{R}^n$ and $-\xi\in\mathbb{R}^n$ respectively, 
that is, $z=x-i\xi$.   
\par
In what follows we state our results. 
Someone might be wondering, what is $\mathcal{B}_h$ in the first place? 
The definition of $\mathcal{B}_h$ 
comes from the Radon transform $\mathcal{R}$ and 
the Bargmann transform $\mathcal{T}_h$ on $\mathbb{R}^n$. 
The relationship between 
$\mathcal{T}_h$, $\mathcal{R}$ and $\mathcal{B}_h$ 
are the following. 
\begin{theorem}
\label{theorem:br}
For $u \in \mathscr{E}^\prime(\mathbb{R}^n)$, 
if $n$ is odd then $\mathcal{T}_hu=\mathcal{B}_h\mathcal{R}u$, 
and if $n$ is even then $\mathcal{T}_hu=\mathcal{B}_h\mathcal{H}\mathcal{R}u$, 
where $\mathcal{H}$ is the Hilbert transform in $t\in\mathbb{R}$ defined by
$$
\mathcal{H}U(\omega,t)
:=
\operatorname{PV}
\int_{-\infty}^\infty
\frac{U(\omega,s)}{t-s}
ds
=
\lim_{\varepsilon\downarrow0}
\int_{\lvert{t-s}\rvert>\varepsilon}
\frac{U(\omega,s)}{t-s}
ds.
$$
Note that if $U(-\omega,-t)=U(\omega,t)$, 
then $\mathcal{H}U(-\omega,-t)=-\mathcal{H}U(\omega,t)$. 
\end{theorem}
\par
It is well-known that the Bargmann transform $\mathcal{T}_h$ 
is an elliptic Fourier integral operator acting on 
$\mathscr{S}^\prime(\mathbb{R}^n)$. 
See, e.g., \cite{Martinez} and \cite{HitrikSjoestrand} for this. 
Similarly $\mathcal{B}_h$ is also an elliptic Fourier integral operator.  
\begin{theorem}
\label{theorem:bh} 
The integral transform $\mathcal{B}_h$ 
is an analytic and elliptic Fourier integral operator on 
$\widetilde{\mathbb{P}^n}$ 
in the sense of {\rm \cite[Section~2.6]{HitrikSjoestrand}}. 
More precisely, 
\begin{itemize}
\item 
The imaginary part of the phase function 
$\operatorname{Im}\phi(x-i\xi,\omega,t)$ 
has non-degenerate critical points given by 
$x\omega=t$, $\xi\ne0$ and $\omega=\pm\xi/\lvert\xi\rvert$. 
Moreover $\phi(x-i\xi,\omega,t)$ satisfies 
$\phi^\prime_\zeta \in \mathbb{R}^n$, 
$\operatorname{Im}\phi^{\prime\prime}_{\zeta\zeta}>0$ 
and 
$\det\phi^{\prime\prime}_{z\zeta}\ne0$ 
at these points.  
\item 
We have $a_n(z,\omega,t)\ne0$ at the non-degenerate critical points 
provided that $\lvert\xi\rvert$ is sufficiently large. 
\item 
The imaginary part of the phase function $\operatorname{Im}\phi(z,\omega,t)$ 
has degenerate critical points given by $x\omega=t$ and $\xi\omega=0$. 
\item 
The contribution of $U \in \mathscr{S}^\prime(\widetilde{\mathbb{P}^n})$ near the degenerate critical points to $\mathcal{B}_hU$ is negligible. 
More precisely, for any $U \in \mathscr{S}^\prime(\widetilde{\mathbb{P}^n})$ 
and for any $(x_0,\xi_0) \in T^\ast(\mathbb{R}^n)\setminus0$, 
there exist a small neighborhood $W$ of $(x_0,\xi_0)$ in $T^\ast(\mathbb{R}^n)\setminus0$ 
and a compactly supported smooth function $\chi(\omega,t)$ on $\mathbb{P}^n$ such that 
all the degenerate critical points for $(x,\xi){\in}W$ are not contained in 
$\operatorname{supp}[1-\chi]$ and 
$$
e^{-\pi\xi^2/h}\mathcal{B}_h(\chi{U})(x-i\xi)
=
\mathcal{O}(e^{-\pi\xi_0^2/8h})
$$
as $h\downarrow0$ uniformly for $(x,\xi){\in}W$. 
\end{itemize}
\end{theorem}
We call the operator $\mathcal{B}_h$ a Bargmann transform on $\widetilde{\mathbb{P}^n}$ 
with a phase function $\phi(z,\omega,t)$ and an amplitude $a_n(z,\omega,t)$. 
Following \cite{Sjoestrand1} 
we can introduce an analytic wave front set and a wave front set for 
$\mathscr{S}^\prime(\widetilde{\mathbb{P}^n})$.  
See also 
\cite{Sjoestrand2}, 
\cite{HitrikSjoestrand} 
and 
\cite{Martinez}. 
The cotangent space and bundle of $\mathbb{P}^n$ are given by 
$$
T^\ast_{(\omega,t)}(\mathbb{P}^n)
=
\{(\eta,\tau)\in\mathbb{R}^n\times\mathbb{R}\ : \ \eta\in\omega^\perp\},
$$
$$
T^\ast(\mathbb{P}^n)
=
\{(\omega,t,\eta,\tau)\ : \ (\omega,t)\in\mathbb{P}^n, \eta\in\omega^\perp, \tau\in\mathbb{R}\}
$$
respectively. 
Our Bargmann transform $\mathcal{B}_h$ is a Fourier integral operator associated with 
the canonical transform $\kappa_B$ defined by 
$$
\kappa_B : 
T^\ast(\mathbb{P}^n)
\ni
\bigl(\omega,t,-\phi^\prime_\zeta(z,\omega,t)\bigr) 
\mapsto 
\bigl(z,\phi^\prime_z(z,\omega,t)\bigr)\in\mathbb{C}^n\times\mathbb{C}^n. 
$$
We have 
\begin{align}
  \phi(x-i\xi,\omega,t)
& =
  i\pi
  \{(x\omega-t)-i\xi\omega\}^2
\\
& =
  i\pi
  \{(x\omega-t)^2-(\xi\omega)^2\}
  +
  2\pi
  (x\omega-t)(\xi\omega), 
\nonumber
\\
  \operatorname{Im}\phi(x-i\xi,\omega,t)
& =
  \pi(x\omega-t)^2-\pi(\xi\omega)^2. 
\label{equation:imphi} 
\end{align}
If we set 
$$
\Phi(z)
:=
\sup_{(\omega,t)\in\mathbb{P}^n}
\bigl\{-\operatorname{Im}\phi(z,\omega,t)\bigr\},
$$
then we deduce that for $\xi\ne0$ 
\begin{align*} 
  -\operatorname{Im}\phi(x-i\xi,\omega,t)
& =
  \pi
  \bigl\{(\xi\omega)^2-(x\omega-t)^2\bigr\}
  \leqq
  \pi(\xi\omega)^2
\\
& \leqq
  \pi\xi^2
  =
  -\operatorname{Im}\phi(z,\xi/\lvert\xi\rvert,x\xi/\lvert\xi\rvert),
\end{align*}
and for $\xi=0$ 
$$
-\operatorname{Im}\phi(x,\omega,t)
=
-
\pi
(x\omega-t)^2
\leqq
0
=
-\operatorname{Im}\phi(x,\omega,x\omega). 
$$
Then we obtain 
$$
\Phi(x-i\xi)
=
\pi\xi^2, 
\quad
\text{i.e.,}
\quad
\Phi(z)
=
\pi(\operatorname{Im}z)^2
=
-
\frac{\pi(z-\bar{z})^2}{4}
$$
for any $\xi\in\mathbb{R}^n$. 
We can characterize the image of $\kappa_B$ using $\Phi(z)$. 
Set 
$$
\Lambda_\Phi
:=
\bigl\{
\bigl(z,-2i\Phi^\prime_z(z)\bigl) 
\ : \ 
z\in\mathbb{C}^n
\bigr\}
=
\{(x-i\xi,2\pi\xi)\ : \ (x,\xi)\in\mathbb{R}^n\times\mathbb{R}^n\}.
$$ 
It is known that 
$\kappa_B\bigl(T^\ast(\mathbb{P}^n)\bigr)=\Lambda_\Phi$ 
holds. 
See 
\cite{Sjoestrand1}, 
\cite{Sjoestrand2} 
and 
\cite{HitrikSjoestrand}. 
In our case we remark that 
$$
\left(
\frac{\xi}{\lvert\xi\rvert},
\frac{x\xi}{\lvert\xi\rvert}
\right)
=
-
\left(
\frac{\xi}{\lvert\xi\rvert},
\frac{x\xi}{\lvert\xi\rvert}
\right)
\quad
\text{in}
\quad
\mathbb{P}^n.
$$
for any $(x,\xi) \in T^\ast(\mathbb{R}^n)\setminus0$. 
We have more about the canonical transform as follows. 
\begin{theorem}
\label{theorem:kb}
We have  
$$
\kappa_B(\omega,t,2\pi\tau\eta,2\pi\tau)
=
(t\omega-\eta-i\tau\omega,2\pi\tau\omega) 
$$
for 
$(\omega,t,2\pi\tau\eta,2\pi\tau) \in T^\ast(\mathbb{P}^n)$, 
and 
$$
\kappa_B^{-1}(x-i\xi,2\pi\xi)
=
\pm
\left(
\frac{\xi}{\lvert\xi\rvert},
\frac{x\xi}{\lvert\xi\rvert}, 
2\pi\lvert\xi\rvert
\left(-x+\frac{(x\xi)\xi}{\lvert\xi\rvert^2}\right),
2\pi\lvert\xi\rvert
\right)
$$ 
for 
$(x,\xi) \in T^\ast(\mathbb{R}^n)\setminus0$.  
\end{theorem}
Suppose that all the distributions in the present paper 
are independent of the semiclassical parameter $h>0$. 
Applying the theories developed in 
\cite{Sjoestrand1},  
\cite{Sjoestrand2}, 
\cite{HitrikSjoestrand} 
and 
\cite{Martinez} 
to Theorem~\ref{theorem:bh} and Theorem~\ref{theorem:kb}, 
we obtain the characterization of 
the wave front set and the analytic wave front set 
of distributions on $\widetilde{\mathbb{P}^n}$.  
\begin{proposition}
\label{theorem:wf} 
Let $(x_0,\xi_0) \in T^\ast(\mathbb{R}^n)\setminus0$, 
and let $U \in \mathscr{S}^\prime(\widetilde{\mathbb{P}^n})$. 
\begin{itemize}
\item 
Let $\operatorname{WF}(U)$ be the wave front set of $U$. 
We have 
$$
\pm
\left(
\frac{\xi_0}{\lvert\xi_0\rvert},
\frac{x_0\xi_0}{\lvert\xi_0\rvert}, 
2\pi\lvert\xi_0\rvert
\left(-x_0+\frac{(x_0\xi_0)\xi_0}{\lvert\xi_0\rvert^2}\right),
2\pi\lvert\xi_0\rvert
\right)
\not\in
\operatorname{WF}(U)
$$
if and only if 
\begin{equation}
e^{-\pi(\operatorname{Im}z)^2/h}
\mathcal{B}_hU(z)
=
\mathcal{O}(h^\infty)
\label{equation:wf1} 
\end{equation}
as $h\downarrow0$ uniformly near $x_0-i\xi_0$, 
that is, there exists a neighborhood $V_0$ at $x_0-i\xi_0$ in $\mathbb{C}^n$ 
such that for any $N=1,2,3,\dotsc$ there exists $C_N>0$ such that 
$$
e^{-\pi(\operatorname{Im}z)^2/h}
\lvert{\mathcal{B}_hU(z)}\vert
\leqq
C_Nh^N
\quad
\text{for}
\quad
z \in V_0. 
$$
\item 
Let $\operatorname{WF}_\text{{\rm A}}(U)$ 
be the analytic wave front set of $U$. 
We have 
$$
\pm
\left(
\frac{\xi_0}{\lvert\xi_0\rvert},
\frac{x_0\xi_0}{\lvert\xi_0\rvert}, 
2\pi\lvert\xi_0\rvert
\left(-x_0+\frac{(x_0\xi_0)\xi_0}{\lvert\xi_0\rvert^2}\right),
2\pi\lvert\xi_0\rvert
\right)
\not\in
\operatorname{WF}_\text{{\rm A}}(U)
$$
if and only if there exists $\varepsilon>0$ such that 
$$
e^{-\pi(\operatorname{Im}z)^2/h}
\mathcal{B}_hU(z)
=
\mathcal{O}(e^{-\varepsilon/h})
$$
as $h\downarrow0$ uniformly near $x_0-i\xi_0$, 
that is, 
there exists a neighborhood $V_0$ at $x_0-i\xi_0$ in $\mathbb{C}^n$ 
and $C_0>0$ such that 
$$
e^{-\pi(\operatorname{Im}z)^2/h}
\lvert{\mathcal{B}_hU(z)}\vert
\leqq
C_0e^{-\varepsilon/h}
\quad
\text{for}
\quad
z \in V_0. 
$$
\end{itemize}
\end{proposition}
The integral transform $\mathcal{T}_h$ is an elliptic Fourier integral operator 
associated with a canonical transform 
$$
K_\mathcal{T}(x,2\pi\xi)
=
\bigl(x-i\xi,-2i\Phi^\prime_z(x-i\xi)\bigr),
\quad
K_\mathcal{T}\bigl(T^\ast(\mathbb{R}^n)\bigr)
=
\Lambda_\Phi,  
$$
and the same results as Proposition~\ref{theorem:wf} hold 
for the wave front set and the analytic wave front set of 
$u \in \mathscr{S}^\prime(\mathbb{R}^n)$. 
\par
In terms of $T^\ast(\mathbb{P}^n)$, 
we have  
$$
(\omega_0,t_0,2\pi\tau_0\eta_0,2\pi\tau_0) 
\not\in\operatorname{WF}(U)
$$
if and only if 
\eqref{equation:wf1} holds uniformly near 
$t_0\omega_0-\eta_0-i\tau_0\omega_0$ in $\mathbb{C}^n$. 
The case of $\operatorname{WF}_\text{A}(U)$ is also stated in the same way. 
Proposition~\ref{theorem:wf} simply states some facts on $\operatorname{WF}(U)$ 
since the definition of the wave front set is not concerned with $\mathcal{B}_h$. 
However, the statement on $\operatorname{WF}_\text{A}(U)$ 
should be interpreted as the comprehensive definition 
of the analytic wave front set of $U$. 
It is well-known that 
the definition of analytic wave front sets has a long history. 
See, e.g., \cite[Section~2.1]{HitrikSjoestrand} for this.  
See also \cite[Chapter VIII]{Hoermander} 
for the standard definition of wave front sets, and 
\cite[Chapter IX]{Hoermander} 
for another definition of analytic wave front sets. 
\par
The Hilbert transform $\mathcal{H}$ does not affect the wave front sets. 
Indeed, $\mathcal{H}$ is a Fourier multiplier with a symbol $-i\pi\operatorname{sgn}(\tau)$, 
that is, $-i\pi$ for $\tau>0$ and $i\pi$ for $\tau<0$. 
See \eqref{equation:hilbert}. 
Let $p(2\pi\tau)$ be a smooth function such that 
$p(2\pi\tau)=\mp{i}\pi$ for $\pm\tau\geqq1$. 
Then, $p(D_t)$ is an elliptic pseudodifferential operator of order zero, 
and $p(D_t)U-\mathcal{H}U$ is analytic in $t$. 
Hence we deduce that the (analytic) wave front set of $U$ is equal to that of $p(D_t)U$ 
and also equal to that of $\mathcal{H}U$. 
So we have automatically the following. 
\begin{theorem}
\label{theorem:wfr} 
Let $(x_0,\xi_0) \in T^\ast(\mathbb{R}^n)\setminus0$, 
and let $u \in \mathscr{E}^\prime(\mathbb{R}^n)$. 
\begin{itemize}
\item 
The point $(x_0,\xi_0)$ belongs to $\operatorname{WF}(u)$ if and only if 
$$
\pm
\left(
\frac{\xi_0}{\lvert\xi_0\rvert},
\frac{x_0\xi_0}{\lvert\xi_0\rvert}, 
2\pi\lvert\xi_0\rvert
\left(-x_0+\frac{(x_0\xi_0)\xi_0}{\lvert\xi_0\rvert^2}\right),
2\pi\lvert\xi_0\rvert
\right)
\in
\operatorname{WF}(\mathcal{R}u). 
$$
\item 
The point $(x_0,\xi_0)$ belongs to $\operatorname{WF}_\text{{\rm A}}(u)$ if and only if  
$$
\pm
\left(
\frac{\xi_0}{\lvert\xi_0\rvert},
\frac{x_0\xi_0}{\lvert\xi_0\rvert}, 
2\pi\lvert\xi_0\rvert
\left(-x_0+\frac{(x_0\xi_0)\xi_0}{\lvert\xi_0\rvert^2}\right),
2\pi\lvert\xi_0\rvert
\right)
\in
\operatorname{WF}_\text{{\rm A}}(\mathcal{R}u).
$$
\end{itemize}
\end{theorem}
It is well-known that the Radon transform is 
an elliptic Fourier integral operators. 
See the celebrated book \cite{GuilleminSternberg} 
of Guillemin and Sternberg for this. 
Quinto, and Krishnan and Quinto gave comprehensive explanation 
for the case of $n=2$ in \cite{Quinto1} and \cite{Quinto2},
and \cite{KrishnanQuinto} respectively. 
Theorem~\ref{theorem:wfr} also give elementary proof of these facts. 
\par
Note that 
$\mathcal{B}_h=\mathcal{T}_h\circ\mathcal{R}^\ast\circ(\partial/\partial t)^{n-1}$,  
and both the Bargmann transform $\mathcal{T}_h$ 
and the Radon transform $\mathcal{R}$ are elliptic Fourier integral operators. 
So Theorem~\ref{theorem:bh} is natural and not surprising. 
The author tried to obtain results not only for the Radon transform 
but also for the $d$-plane transform which is a mapping of functions to 
their integrations on all the $d$-dimensional affine subspaces. 
See \cite{Helgason} for this. 
Unfortunately, he could not obtain the results for $d{\ne}n-1$. 
He believes that Theorem~\ref{theorem:bh} is meaningful 
as microlocal analysis on the space of hyperplanes, 
which is the set of the variables of measurements. 
Recently microlocal analysis plays a crucial role in imaging science. 
In fact, for example, 
microlocal artifacts and limited data problems for CT scanners 
are very important subjects in geometric tomography, 
and studied based on microlocal analysis. 
See \cite{KrishnanQuinto} and \cite{PalaciosUhlmannWang} for instance. 
He expects $\mathcal{B}_h$ to play an important role 
in pure and applied mathematics related to geometric tomography, 
because $\mathcal{B}_h$ maps the microlocal singularities of the measurements 
to the microlocal singularities of the objects to be examined, 
based on the canonical relation. 
\par
All we have to do is to prove 
Theorems~\ref{theorem:br}, \ref{theorem:bh} and \ref{theorem:kb}. 
The plan of the present paper is as follows. 
In Section~\ref{section:origin} 
we recall some basic facts on $\mathcal{T}_h$ 
and prove Theorem~\ref{theorem:br}. 
In Section~\ref{section:FIO} we prove Theorems~\ref{theorem:bh}. 
Finally in Section~\ref{section:WF} 
we find the canonical transform $K_{\mathcal{B}}$ 
of the form which is independent of the local coordinates of $\mathbb{P}^n$ 
to prove Theorem~\ref{theorem:kb}.    
\section{The origin of our integral transform}
\label{section:origin} 
In the present section we begin by recalling some basic facts 
on the $h$-Fourier transform and $\mathcal{T}_h$, 
and prove Theorem~\ref{theorem:br}. 
See, e.g., \cite{Martinez} for $h$-Fourier transform. 
The $h$-Fourier transform of $u \in \mathscr{S}(\mathbb{R}^n)$ is defined by 
$$
\mathcal{F}_hu(\eta)
:=
h^{-n/2}
\int_{\mathbb{R}^n}
e^{-2\pi i y\eta/h}
u(y)
dy,
\quad
\eta\in\mathbb{R}^n, 
$$
and the inverse $h$-Fourier transform of $u$ is defined by 
$\mathcal{F}_h^\ast u(y):=\mathcal{F}_hu(-y)$. 
It is well-known that 
$\mathcal{F}_h[e^{-\pi y^2/h}](\eta)=e^{-\pi\eta^2/h}$ 
for instance. 
The integral transforms $\mathcal{F}_h$ and $\mathcal{F}_h^\ast$ 
are the inverse operators of each other, that is, 
$\mathcal{F}_h^\ast\mathcal{F}_hu=\mathcal{F}_h\mathcal{F}_h^\ast u=u$ 
holds for any $u \in \mathscr{S}(\mathbb{R}^n)$. 
We introduce the Heisenberg inequality and the coherent states 
to explain what $\mathcal{T}_h$ detects. 
The $L^2$-inner product and $L^2$-norm on $\mathbb{R}^n$ are defined by 
$$
(u,v)_{L^2(\mathbb{R}^2)}
:=
\int_{\mathbb{R}^n}
u(y) \overline{v(y)}
dy,
\quad
\lVert{u}\rVert_{L^2(\mathbb{R}^n)}
:=
\sqrt{(u,u)_{L^2(\mathbb{R}^n)}} 
$$
for $u,v \in \mathscr{S}(\mathbb{R}^n)$. 
The semiclassical Heisenberg inequality in our setting of $h$-Fourier transform is as follows. 
\begin{proposition}
\label{theorem:heisenberg}
Suppose that $u \in \mathscr{S}(\mathbb{R}^n)$ 
and $\lVert{u}\rVert_{L^2(\mathbb{R}^n)}=1$. 
Set 
$$
x
:=
\int_{\mathbb{R}^n}
y\lvert{u(y)}\rvert^2
dy,
\quad
\xi
:=
\int_{\mathbb{R}^n}
\eta\lvert{\mathcal{F}_hu(\eta)}\rvert^2
d\eta
=
\int_{\mathbb{R}^n}
\frac{h}{2\pi i}
\frac{\partial u}{\partial y}(y)
\overline{u(y)}
dy,
$$
\begin{align*}
  \sigma(u,X)^2
& :=
  \int_{\mathbb{R}^n}
  \lvert{2\pi(y-X)u(y)}\rvert^2
  dy,
\\
  \sigma(\hat{u},\Xi)^2
& :=
  \int_{\mathbb{R}^n}
  \lvert{2\pi(\eta-\Xi)\mathcal{F}_hu(\eta)}\rvert^2
  d\eta
  =
  \int_{\mathbb{R}^n}
  \lvert{(hD_y-2\pi\Xi)u(y)}\rvert^2
  dy, 
\end{align*}
where 
$X,\Xi\in\mathbb{R}^n$ and $D_y=-i\partial/\partial y$. 
Then we have 
\begin{equation}
\sigma(u,x)
\cdot
\sigma(\hat{u},\xi) 
\geqq
n\pi h, 
\quad
\sigma(u,X)
\geqq
\sigma(u,x),
\quad
\sigma(\hat{u},\Xi)
\geqq
\sigma(\hat{u},\xi)
\label{equation:heisenberg} 
\end{equation}
for any $X,\Xi\in\mathbb{R}^n$.
\end{proposition}
\begin{proof}
The proof is essentially given in \cite[Exercise 9 in Page 58]{Martinez}. 
\end{proof} 
The first inequality of \eqref{equation:heisenberg} 
is said to be the semiclassical Heisenberg inequality. 
Consider the case that the equality holds for this. 
\begin{definition}
\label{theorem:coherent1} 
Fix arbitrary $(x,\xi)\in T^\ast(\mathbb{R}^n)$. 
The function of $y \in\mathbb{R}^n$ given by 
$$
\psi(y;x,\xi,h)
:=
2^{n/4}
h^{-n/4}
e^{-2\pi i (x-y)\xi/h-\pi(x-y)^2/h}
$$
belongs to $\mathscr{S}(\mathbb{R}^n)$,  
and is said to be the coherent state at $(x,\xi)$. 
\end{definition} 
The properties of the coherent states are the following. 
\begin{proposition}
\label{theorem:coherent2} 
\quad
\begin{itemize}
\item 
$\mathcal{F}_h\psi(\cdot;x,\xi,h)(\eta)=2^{n/4}h^{-n/4}e^{-2\pi i x\eta/h-\pi(\xi-\eta)^2/h}$,  
$\lVert{\psi(\cdot;x,y,h)}\rVert_{L^2(\mathbb{R}^n)}=1$, 
$$
\int_{\mathbb{R}^n}
y\lvert{\psi(y;x,\xi,h)}\rvert^2
dy
=
x,
\quad
\int_{\mathbb{R}^n}
\eta\lvert{\mathcal{F}_h\psi(\cdot;x,\xi,h)(\eta)}\rvert^2
d\eta
=
\xi.
$$
\item 
The coherent state attains the equality of 
the first inequality in {\rm \eqref{equation:heisenberg}}. 
More precisely, 
$$
\sigma(\psi(\cdot;x,\xi,h),x)^2
=
\sigma(\hat{\psi}(\cdot;x,\xi,h),\xi)^2
=
n\pi h.
$$
\end{itemize}
\end{proposition}
\begin{proof} 
The proof is done by elementary computation. 
Most of them can be reduced to 
$$
\int_{\mathbb{R}^n}e^{-\pi y^2}dy=1. 
$$
We omit the detail. 
See \cite[Chapter~3]{Martinez}.  
\end{proof}
It is worth mentioning that 
\begin{align*}
  \lvert{\psi(y;x,\xi,h)}\rvert
& =
  2^{n/4}h^{-n/4}
  e^{-\pi(x-y)^2/h},
\\
  \lvert{\mathcal{F}_h\psi(\cdot;x,\xi,h)(\eta)}\rvert
& =
  2^{n/4}h^{-n/4}
  e^{-\pi(\xi-\eta)^2/h}, 
\\
  e^{-\pi\xi^2/h}
  \mathcal{T}_hu(x-i\xi)
& =
  h^{-n/2}
  \int_{\mathbb{R}^n}
  u(y)
  \overline{\psi(y;x,\xi,h)}
  dy
\\
& =
  h^{-n/2}
  \int_{\mathbb{R}^n}
  \mathcal{F}_hu(\eta)
  \overline{\mathcal{F}_h\psi(\cdot;x,\xi,h)(\eta)}
  d\eta. 
\end{align*}
The Bargmann transform is 
not only a function rearranged by a certain distribution function, 
but also its $h$-Fourier transform rearranged by a certain distribution function. 
Roughly speaking, 
the position $y$ follows some modified normal distribution with 
mean $x$ and standard deviation $\sqrt{h/2\pi}$, 
and the frequency $\eta$ follows some modified normal distribution with 
mean $\xi$ and standard deviation $\sqrt{h/2\pi}$.  
In other words the Bargmann transform $\mathcal{T}_hu(x,\xi)$ detects 
the microlocal information of $u$ in the Heisenberg box of the form
$$
\left\{y\in\mathbb{R}^n : \lvert{y-x}\rvert\leqq\sqrt{\frac{h}{2\pi}} \right\}
\times
\left\{\eta\in\mathbb{R}^n : \lvert{\eta-\xi}\rvert\leqq\sqrt{\frac{h}{2\pi}} \right\}
$$
at each point $(x,\xi) \in T^\ast(\mathbb{R}^n)$. 
\par
In what follows we prove Theorem~\ref{theorem:br}. 
This follows from the Plancherel formula of the Radon transform $\mathcal{R}$. 
\begin{proposition}[{{\bf The Plancherel Formula} \cite[Page 26]{Helgason}}]
\label{theorem:plancherel} 
For $u,v \in \mathscr{S}(\mathbb{R}^n)$, 
$$
(-1)^{(n-1)/2}C_n
\int_{\mathbb{R}^n}
u(x) 
v(x) 
dx 
= 
\iint_{\mathbb{P}^n}
\lvert{D_t}\rvert^{n-1}\mathcal{R}u(\omega,t)
\cdot
\mathcal{R}v(\omega,t)
d\omega
dt,
$$
where $D_t=-i\partial/\partial t$ 
and $C_n$ is the constant defined by {\rm \eqref{equation:constant}}. 
\end{proposition} 
Here we begin to prove Theorem~\ref{theorem:br}. 
\begin{proof}[{\bf Proof of Theorem~\ref{theorem:br}}] 
Let $k$ be a positive integer, 
and let $U(\omega,t)$ belong to $\mathscr{S}(\mathbb{P}^n)$.  
For $n=2k+1$ 
\begin{equation}
\lvert{D_t}\rvert^{n-1}U(\omega,t)
=
\lvert{D_t}\rvert^{2k}U(\omega,t)
=
\left(-\frac{\partial^2}{\partial t^2}\right)^kU(\omega,t)
=
(-1)^k
\frac{\partial^{2k}U}{\partial t^{2k}}(\omega,t), 
\label{equation:dt1}
\end{equation}
and for $n=2k$ 
\begin{align}
  \lvert{D_t}\rvert^{n-1}U(\omega,t)
& =
  \lvert{D_t}\rvert^{2k-1}U(\omega,t)
\nonumber
\\
&  =
  \frac{(-1)^{k-1}}{\pi}
  \frac{\partial^{2k-1}}{\partial t^{2k-1}}
  \operatorname{PV}
  \int_{-\infty}^\infty
  \frac{U(\omega,s)}{t-s}
  ds
\nonumber
\\
& =
  \frac{(-1)^{k-1}}{\pi}
  \frac{\partial^{2k-1}}{\partial t^{2k-1}}
  \mathcal{H}U(\omega,t). 
\label{equation:dt2}
\end{align}
Indeed, the odd case is obvious, 
and the even case follows from the facts that 
the symbol of $\lvert{D_t}\rvert^{2k-1}$ is 
\begin{align*}
  e^{2\pi i t\tau}\lvert{D_t}\rvert^{2k-1}e^{-2\pi i t\tau}
& =
  \lvert{2\pi\tau}\rvert^{2k-1}
  =
 (2\pi i \tau)^{2k-1}(-i)^{2k-1} \operatorname{sgn}(\tau)
\\
& =
  (-1)^{k-1} \{-i \operatorname{sgn}(\tau)\} (2\pi i \tau)^{2k-1} 
\end{align*}
and the usual Fourier transform of the principal value of $1/t$ is 
\begin{equation}
\mathcal{F}_1
\left[\operatorname{PV}\frac{1}{t}\right](\tau)
=
-i\pi\operatorname{sgn}(\tau)
\quad
\text{in}
\quad
\mathscr{S}^\prime(\mathbb{R}).
\label{equation:hilbert}
\end{equation}
\par
Applying Proposition~\ref{theorem:plancherel} to $\mathcal{T}_hu$, 
we deduce that 
\begin{align*}
  \mathcal{T}_hu(z)
& =
  2^{n/4}h^{-3n/4}
  \int_{\mathbb{R}^n}
  e^{-\pi(z-y)^2/h}
  u(y)
  dy
\\
& =
  \frac{2^{n/4}}{(-1)^{(n-1)/2}C_n}
  h^{-3n/4}
  \iint_{\mathbb{P}^n}
  \mathcal{R}[e^{-\pi(z-\cdot)^2/h}](\omega,t)
  \cdot
  \lvert{D_t}\rvert^{n-1}
  \mathcal{R}u(\omega,t)
  d\omega 
  dt.  
\end{align*}
If we set 
$$
\widetilde{\mathcal{R}}u(\omega,t)
:=
\begin{cases}
\mathcal{R}u(\omega,t)
&\ (n=2k+1),
\\
\mathcal{H}\mathcal{R}u(\omega,t)
&\ (n=2k),  
\end{cases}
$$
and use the integration by parts with 
\eqref{equation:dt1} and \eqref{equation:dt2}, 
we have 
$$
\mathcal{T}_hu(z)
 =
A_nh^{-3n/4}
\int_{\mathbb{R}^n}
\left(
\frac{1}{\sqrt{\pi}}
\frac{\partial}{\partial t}
\right)^{n-1}
\mathcal{R}[e^{-\pi(z-\cdot)^2/h}](\omega,t)
\cdot
\widetilde{\mathcal{R}}u(\omega,t)
d\omega 
dt.
$$
So it suffices to show that 
\begin{equation}
\left(
\frac{1}{\sqrt{\pi}}
\frac{\partial}{\partial t}
\right)^{n-1}
\mathcal{R}[e^{-\pi(z-\cdot)^2/h}](\omega,t)
=
e^{-\pi(z\omega-t)^2/h}
H_{n-1}\left(\sqrt{\frac{\pi}{h}}(z\omega-t)\right).
\label{equation:tuas} 
\end{equation}
Firstly we compute the Radon transform of $e^{-\pi(z-y)^2/h}$ in $y$. 
Fix arbitrary $\omega \in \mathbb{S}^{n-1}$, 
and set $z^\prime:=z-(z\omega)\omega$. 
Note that $z^\prime(c\omega)=0$ for any $c\in\mathbb{C}$. 
Then we deduce that 
\begin{align*}
  \mathcal{R}[e^{-\pi(z-\cdot)^2/h}](\omega,t)
& =
  \int_{\omega^\perp}
  e^{-\pi(z-y-t\omega)^2/h}
  dm(y)
\\
&  =
  \int_{\omega^\perp}
  e^{-\pi\{(z^\prime-y)+(z\omega-t)\omega\}^2/h}
  dm(y)
\\
& =
  \int_{\omega^\perp}
  e^{-\pi(z^\prime-y)^2/h-\pi(z\omega-t)^2/h}
  dm(y)
\\
& =
  e^{-\pi(z\omega-t)^2/h}
  \int_{\omega^\perp}
  e^{-\pi(z^\prime-y)^2/h}
  dm(y). 
\end{align*}
Since 
$$
\int_{-\infty}^\infty
e^{-\pi(s-a-ib)^2}
ds
=1
$$
for any $a,b\in\mathbb{R}$, we deduce that 
$$
\mathcal{R}[e^{-\pi(z-\cdot)^2/h}](\omega,t)
=
h^{(n-1)/2}
e^{-\pi(z\omega-t)^2/h}.  
$$
Substitute this into the left hand side of \eqref{equation:tuas}, 
we have 
\begin{align*}
  \left(
  \frac{1}{\sqrt{\pi}}
  \frac{\partial}{\partial t}
  \right)^{n-1}
  \mathcal{R}[e^{-\pi(z-\cdot)^2/h}](\omega,t) 
& =
\left(
\frac{h}{\sqrt{\pi}}
\frac{\partial}{\partial t}
\right)^{n-1}
e^{-\pi(z\omega-t)^2/h}
\\
& =
  \left(-\frac{d}{ds}\right)^{n-1}
  e^{-s^2}
  \bigg\vert_{s=\sqrt{\pi}(z\omega-t)/\sqrt{h}}
\\
& =
  e^{-s^2}H_{n-1}(s)
  \bigg\vert_{s=\sqrt{\pi}(z\omega-t)/\sqrt{h}}
\\
& =
  e^{-\pi(z\omega-t)^2/h}
  H_{n-1}\left(\sqrt{\frac{\pi}{h}}(z\omega-t)\right). 
\end{align*}
This completes the proof. 
\end{proof}
%
%
\section{Holomorphic Fourier integral operators}
\label{section:FIO}
The present section is devoted to proving Theorem~\ref{theorem:bh}. 
We show that $\phi(z,\omega,t)$ is a 
complex-valued phase function with non-degenerate critical points satisfying 
$x\omega=t$, $\xi\ne0$ and $\omega=\xi/\lvert\xi\rvert$.  
We also show that $\operatorname{Im}\phi(z,\omega,t)$ also has degenerate critical points 
but their influence on $\mathcal{B}_h$ is negligible. 
The ellipticity of the amplitude $a_n(z,\omega,t)$ 
at non-degenerate critical points follows automatically. 
Indeed if the assertion for $\phi(z,\omega,t)$ is true, then 
\begin{align*}
  a_n(z,\omega,t)
& =
  H_{n-1}
  \left(-i\sqrt{\frac{\pi}{h}}\lvert\xi\rvert\right)
\\
& =
  (n-1)!
  \sum_{j=0}^{[(n-1)/2]}
  \frac{(-1)^j}{j!(n-1-2j)!}
  \left(-i\sqrt{\frac{\pi}{h}}\lvert\xi\rvert\right)^{n-1-j}
\\
& =
  (-i)^{(n-1)}
  \left(\frac{\pi}{h}\right)^{(n-1)/2}
  \lvert\xi\rvert^{n-1}
  +
  \mathcal{O}(\lvert\xi\rvert^{n-3})
  \ne0 
\end{align*}
at the non-degenerate critical points provided that $\lvert\xi\rvert$ is sufficiently large. 
This shows that $\mathcal{B}_h$ is an elliptic Fourier integral operator. 
\par
In what follows we will work only for $\omega$ near ${}^t[0,\dotsc,0,\pm1]$. 
Set 
$$
\omega
=
\begin{bmatrix}
\omega^\prime \\ \omega_n 
\end{bmatrix}
\in
\mathbb{S}^{n-1},
\quad
\omega^\prime
=
\begin{bmatrix}
\omega_1 \\ \vdots \\ \omega_{n-1} 
\end{bmatrix}
\in
\mathbb{R}^{n-1}, 
\quad
$$
and suppose $(\omega^\prime)^2=\omega_1^2+\dotsb+\omega_{n-1}^2<1/2$. 
Then $\omega_n=\pm\sqrt{1-(\omega^\prime)^2}$. 
In other words, we employ $\zeta=(\omega^\prime,t)$ 
as local coordinates of $\mathbb{P}^n$. 
The plan of our proof of Theorem~\ref{theorem:bh} is as follows. 
\begin{itemize}
\item
Step~1.
\quad
We find the critical points of $\operatorname{Im}\phi(x-i\xi,\omega,t)$. 
As a result we will obtain two cases: 
\begin{itemize}
\item
Case~I.
\quad
$x\omega=t$ and $\xi\omega=0$. 
\item
Case~II.
\quad
$x\omega=t$, 
$\xi\ne0$ 
and 
$(\xi\omega)^\prime_{\omega^\prime}:=\dfrac{\partial}{\partial\omega^\prime}\xi\omega=0$, 
which imply 
$x\omega=t$, 
$\xi\ne0$ 
and 
$\omega=\pm\xi/\lvert\xi\rvert$. 
\end{itemize}
\item
Step~2,
\quad
We check the non-degeneracy $\operatorname{Im}\phi^{\prime\prime}_{\zeta\zeta}>0$ 
at the critical points. 
Case~I is concerned with degenerate critical points.  
The condition of Case~II is the necessary and sufficient condition 
of the non-degenerate critical points of $\phi(z,\omega,t)$. 
\item
Step~3. 
\quad
We show $\det\phi^{\prime\prime}_{z\zeta}\ne0$ at the points of Case~II.  
\item 
Step~4.
\quad
We show that the contribution of $U \in \mathscr{S}^\prime(\widetilde{\mathbb{P}^n})$ near the degenerate critical points to $\mathcal{B}_hU$ is negligible. 
\end{itemize}
\begin{proof}[{\bf Step~1 of Proof of Theorem~\ref{theorem:bh}}] 
We shall find the critical points of $\operatorname{Im}\phi(x-i\xi,\omega,t)$. 
Using \eqref{equation:imphi}, we have 
\begin{align*}
  \operatorname{Im}\phi(x-i\xi,\omega,t)
& =
  \pi(x\omega-t)^2-\pi(\xi\omega)^2
\\
& =
  \pi
  \left\{
  \sum_{j=1}^{n-1}x_j\omega_j 
  \pm 
  x_n\sqrt{1-(\omega^\prime)^2}
  -
  t
  \right\}^2
\\
& -
  \pi
  \left\{
  \sum_{j=1}^{n-1}\xi_j\omega_j 
  \pm 
  \xi_n\sqrt{1-(\omega^\prime)^2}
  \right\}^2.  
\end{align*}
Then we have 
\begin{align*}
& \operatorname{Im}\phi^\prime_\zeta(x-i\xi,\omega,t)
\\
  =
& 2\pi(x\omega-t)
  \frac{\partial}{\partial\zeta}(x\omega-t)
  -
  2\pi(\xi\omega)
  \frac{\partial}{\partial\zeta}(\xi\omega)
\\
  =
& 2\pi(x\omega-t)
  \begin{bmatrix}
  (x\omega)^\prime_{\omega^\prime}
  \\
  -1 
  \end{bmatrix}
  -2\pi(\xi\omega)
  \begin{bmatrix}
  (\xi\omega)^\prime_{\omega^\prime}
  \\
  0 
  \end{bmatrix}
\\
  =
& 2\pi(x\omega-t)
  \begin{bmatrix}
  x^\prime\mp\dfrac{x_n}{\sqrt{1-(\omega^\prime)^2}}\omega^\prime 
  \\
  -1
  \end{bmatrix}  
  -
  2\pi(\xi\omega)
  \begin{bmatrix}
  \xi^\prime\mp\dfrac{\xi_n}{\sqrt{1-(\omega^\prime)^2}}\omega^\prime 
  \\
  0
  \end{bmatrix}.
\end{align*}
The $n$-th row of the right hand side of the above is $-2\pi(x\omega-t)$. 
This shows that the critical points must satisfy $x\omega=t$. 
If $x\omega=t$, then
$$
\operatorname{Im}\phi^\prime_\zeta(x-i\xi,\omega,t)
=
-
2\pi(\xi\omega)
\begin{bmatrix}
(\xi\omega)^\prime_{\omega^\prime}
\\
0
\end{bmatrix}
=
-
2\pi(\xi\omega)
\begin{bmatrix}
\xi^\prime\mp\dfrac{\xi_n}{\sqrt{1-(\omega^\prime)^2}}\omega^\prime 
\\
0
\end{bmatrix}. 
$$
Hence the critical points must satisfy 
$\xi\omega=0$ or $(\xi\omega)^\prime_{\omega^\prime}=0$  
under the condition $x\omega=t$. 
Note that $\xi=0$ can be seen as the special case of $\xi\omega=0$, 
and that the critical points must satisfy Case~I or Case~II. 
Here we see the detail of Case~II: $\xi\ne0$ and $(\xi\omega)^\prime_{\omega^\prime}=0$. 
Recall that $\omega_n=\pm\sqrt{1-(\omega^\prime)^2}$ near $\pm1$. 
Then we have 
$$
0
=
(\xi\omega)^\prime_{\omega^\prime}
=
\xi^\prime-\frac{\xi_n}{\omega_n}\omega^\prime,
\quad
\xi
=
\begin{bmatrix}
\xi^\prime \\ \xi_n 
\end{bmatrix}
=
\frac{\xi_n}{\omega_n}\omega  
$$
and $\omega=\pm\xi/\lvert\xi\rvert$. 
\end{proof}
\begin{proof}[{\bf Step~2 of Proof of Theorem~\ref{theorem:bh}}]
We shall check the non-degeneracy $\operatorname{Im}\phi^{\prime\prime}_{\zeta\zeta}>0$ 
at the points of Case~I and Case~II. 
Since the critical points must satisfy $x\omega=t$, we have 
\begin{align*}
  \operatorname{Im}\phi^{\prime\prime}_{\zeta\zeta}(x-i\xi,\omega,t)
& =
  2\pi
  \begin{bmatrix}
  (x\omega)^\prime_{\omega^\prime}
  \\
  -1 
  \end{bmatrix}
  \begin{bmatrix}
  {}^t(x\omega)^\prime_{\omega^\prime} & -1 
  \end{bmatrix} 
\\
& -
  2\pi
  \begin{bmatrix}
  (\xi\omega)^\prime_{\omega^\prime}{}^t(\xi\omega)^\prime_{\omega^\prime} & 0
  \\
  0 & 0
  \end{bmatrix}
  -
  2\pi
  \begin{bmatrix}
  (\xi\omega)(\xi\omega)^{\prime\prime}_{\omega^\prime \omega^\prime} & 0
  \\
  0 & 0 
  \end{bmatrix} 
\\
& =
  2\pi
  \begin{bmatrix}
  (x\omega)^\prime_{\omega^\prime}{}^t(x\omega)^\prime_{\omega^\prime} - A
  & 
  -(x\omega)^\prime_{\omega^\prime}
  \\ 
  -{}^t(x\omega)^\prime_{\omega^\prime}
  & 
  1
  \end{bmatrix}, 
\\
  A
& =
  (\xi\omega)^\prime_{\omega^\prime}{}^t(\xi\omega)^\prime_{\omega^\prime}
  +
  (\xi\omega)(\xi\omega)^{\prime\prime}_{\omega^\prime \omega^\prime}
\end{align*}
for $x\omega=t$. 
Let $E_{n-1}$ be the $(n-1)\times(n-1)$ identity matrix. Set 
$$
P=P(x,\omega)
:=
\begin{bmatrix}
E_{n-1} & 0
\\
{}^t(x\omega)^\prime_{\omega^\prime} & 1 
\end{bmatrix}. 
$$
Then $\det{P}=1$. 
By using $P(x,\omega)$ we deduce that  
$$
{}^tP(x,\omega) 
  \operatorname{Im}\phi^{\prime\prime}_{\zeta\zeta}(x-i\xi,\omega,t)
  P(x,\omega) 
=
2\pi
\begin{bmatrix}
-A & 0
\\
0 & 1 
\end{bmatrix}
$$
for $x\omega=t$, 
and that $\operatorname{Im}\phi^{\prime\prime}_{\zeta\zeta}>0$ is equivalent to $-A>0$. 
\par
Firstly we will check $-A>0$ for Case~I. 
Suppose that $x\omega=t$ and $\xi\omega=0$. 
Then $-A=-(\xi\omega)^\prime_{\omega^\prime}{}^t(\xi\omega)^\prime_{\omega^\prime}$. 
For any 
$y \in \mathbb{R}^{n-1} \cap \bigl((\xi\omega)^\prime_{\omega^\prime}\bigr)^\perp\setminus\{0\}$, 
${}^ty(-A)y=-\bigl\{y(\xi\omega)^\prime_{\omega^\prime}\bigr\}^2=0$. 
Hence the points of Case~I cannot be the non-degenerate critical points. 
\par
Secondly we will check $-A>0$ for Case~II. 
Suppose that $x\omega=t$ and $(\xi\omega)^\prime_{\omega^\prime}=0$, that is, 
$x\omega=t$, $\xi_n\ne0$ and $\xi=\xi_n\omega/\omega_n$. 
Then $\xi\omega=\xi_n/\omega_n\ne0$, and 
\begin{align*}
  -A
& =
  -
  (\xi\omega)
  (\xi\omega)^{\prime\prime}_{\omega^\prime\omega^\prime}
  =
  -
  \frac{\xi_n}{\omega_n}
  \frac{\partial}{\partial\omega^\prime}(\xi\omega)^\prime_{\omega^\prime}
  =
  -
  \frac{\xi_n}{\omega_n}
  \frac{\partial}{\partial\omega^\prime}
  \left(
  \xi^\prime
  -
  \frac{\xi_n}{\omega_n}
  \omega^\prime
  \right)
\\
& =
  \frac{\xi_n}{\omega_n}
  \frac{\partial}{\partial\omega^\prime}
  \left(
  \frac{\xi_n}{\omega_n}
  \omega^\prime
  \right)
  =
  \pm
  \frac{\xi_n}{\sqrt{1-(\omega^\prime)^2}}
  \frac{\partial}{\partial\omega^\prime}
  \left(
  \pm
  \frac{\xi_n}{\sqrt{1-(\omega^\prime)^2}}
  \omega^\prime
  \right)
\\
& =
  \frac{\xi_n}{\sqrt{1-(\omega^\prime)^2}}
  \frac{\partial}{\partial\omega^\prime}
  \left(
  \frac{\xi_n}{\sqrt{1-(\omega^\prime)^2}}
  \omega^\prime
  \right)
\\
& =
  \frac{\xi_n}{\sqrt{1-(\omega^\prime)^2}}
  \left(
  \frac{\xi_n}{\sqrt{1-(\omega^\prime)^2}}
  E_{n-1}
  -
  \frac{\xi_n\omega^\prime(-2{}^t\omega^\prime)}{2\bigl(1-(\omega^\prime)^2\bigr)^{3/2}}
  \right)
\\
& =
  \frac{\xi_n^2}{1-(\omega^\prime)^2}
  \left(
  E_{n-1}
  +
  \frac{\omega^\prime {}^t\omega^\prime}{1-(\omega^\prime)^2}
  \right)
  =
  \frac{\xi_n^2}{\omega_n^2}
  \left(
  E_{n-1}
  +
  \frac{\omega^\prime {}^t\omega^\prime}{\omega_n^2}
  \right).
\end{align*}
We deduce that for any $y \in \mathbb{R}^{n-1}\setminus\{0\}$, 
$$
{}^ty(-A)y
=
\frac{\xi_n^2}{\omega_n^2}
\left(
\lvert{y}\rvert^2
+
\frac{(y\omega)^2}{\omega_n^2}
\right)
\geqq
\frac{\xi_n^2}{\omega_n^2}\lvert{y}\rvert^2
>0, 
$$
which shows that $-A>0$. 
Thus the non-degenerate critical points are characterized by the conditions of Case~II, 
and $\operatorname{Im}\phi^{\prime\prime}_{\zeta\zeta}>0$ holds at these points. 
\end{proof}
\begin{proof}[{\bf Step~3 of Proof of Theorem~\ref{theorem:bh}}]
We shall show that $\det\phi^{\prime\prime}_{z\zeta}(z,\omega,t)\ne0$ 
at the non-degenerate critical points. We first compute 
$\phi^{\prime\prime}_{z\zeta}(z,\omega,t)$. 
Elementary computation gives 
\begin{align*}
  \phi^\prime_z(z,\omega,t) 
& =
  2\pi i(z\omega-t)\omega,
\\
  \phi^{\prime\prime}_{z\zeta}(z,\omega,t)
& =
  2\pi i 
  \begin{bmatrix}
  \omega^\prime \\ \omega_n 
  \end{bmatrix}
  \begin{bmatrix}
  {}^tz^\prime-z_n{}^t\omega^\prime/\omega_n & -1
  \end{bmatrix}
  +
  2\pi i (z\omega-t)
  \begin{bmatrix}
  E_{n-1} & 0
  \\
  -{}^t\omega^\prime/\omega_n & 0 
  \end{bmatrix}
\\
& =
  2\pi i
  \begin{bmatrix}
  \omega^\prime{}^tz^\prime-z_n\omega^\prime{}^t\omega^\prime/\omega_n 
  & 
  -\omega^\prime
  \\
  \omega_n{}^tz^\prime-z_n{}^t\omega^\prime & -\omega_n 
  \end{bmatrix}
  +
  2\pi i (z\omega-t)
  \begin{bmatrix}
  E_{n-1} & 0
  \\
  -{}^t\omega^\prime/\omega_n & 0 
  \end{bmatrix}. 
\end{align*}
Substitute $x\omega=t$ and 
$\xi^\prime=\xi_n\omega^\prime/\omega_n$ 
into the above. We have 
$$
\phi^{\prime\prime}_{z\zeta}(z,\omega,t)
=
2\pi i
\begin{bmatrix}
\omega^\prime{}^tz^\prime-z_n\omega^\prime{}^t\omega^\prime/\omega_n 
& 
-\omega^\prime
\\
\omega_n{}^tz^\prime-z_n{}^t\omega^\prime & -\omega_n 
\end{bmatrix}
+
2\pi 
\frac{\xi_n}{\omega_n}
\begin{bmatrix}
E_{n-1} & 0
\\
-{}^t\omega^\prime/\omega_n & 0 
\end{bmatrix}. 
$$
Multiplying $\phi^{\prime\prime}_{z\zeta}(z,\omega,t)$ 
by some regular matrices, 
we modify it to see that its determinant does not vanish. 
Set 
$$
L
:=
\begin{bmatrix}
E_{n-1} & -\omega^\prime/\omega_n
\\
0 & 1 
\end{bmatrix}, 
\quad
Q
:=
\begin{bmatrix}
E_{n-1} & 0
\\
{}^tz^\prime & 1 
\end{bmatrix},
\quad
R
:=
\begin{bmatrix}
E_{n-1} & 0
\\
z_n{}^t\omega^\prime/\omega_n & 1 
\end{bmatrix}. 
$$
Then $\det{L}=\det{Q}=\det{R}=1$. 
We deduce that 
\begin{align*}
& L \phi^{\prime\prime}_{z\zeta}(z,\omega,t) QR
\\
  =
& 2\pi
  L
  \left\{
  i
  \begin{bmatrix}
  \omega^\prime{}^tz^\prime-z_n\omega^\prime{}^t\omega^\prime/\omega_n 
  & 
  -\omega^\prime
  \\
  \omega_n{}^tz^\prime-z_n{}^t\omega^\prime & -\omega_n 
  \end{bmatrix}
  +
  \frac{\xi_n}{\omega_n}
  \begin{bmatrix}
  E_{n-1} & 0
  \\
  -{}^t\omega^\prime/\omega_n & 0 
  \end{bmatrix}
  \right\}
  \begin{bmatrix}
  E_{n-1} & 0
  \\
  {}^tz^\prime & 1 
  \end{bmatrix}
  R
\\
  =
& 2\pi
  \begin{bmatrix}
  E_{n-1} & -\omega^\prime/\omega_n
  \\
  0 & 1 
  \end{bmatrix}
  \left\{
  i
  \begin{bmatrix}
  -z_n\omega^\prime{}^t\omega^\prime/\omega_n 
  & 
  -\omega^\prime
  \\
  -z_n{}^t\omega^\prime & -\omega_n 
  \end{bmatrix}
  +
  \frac{\xi_n}{\omega_n}
  \begin{bmatrix}
  E_{n-1} & 0
  \\
  -{}^t\omega^\prime/\omega_n & 0 
  \end{bmatrix}
  \right\}
  R
\\
  =
& 2\pi
  \left\{
  i
  \begin{bmatrix}
  O
  & 
  0
  \\
  -z_n{}^t\omega^\prime & -\omega_n 
  \end{bmatrix}
  +
  \frac{\xi_n}{\omega_n}
  \begin{bmatrix}
  E_{n-1}+\omega^\prime{}^t\omega^\prime/\omega_n^2 & 0
  \\
  -{}^t\omega^\prime/\omega_n & 0 
  \end{bmatrix}
  \right\}
  \begin{bmatrix}
  E_{n-1} & 0
  \\
  z_n{}^t\omega^\prime/\omega_n & 1 
  \end{bmatrix}
\\
  =
& 2\pi
  \left\{
  i
  \begin{bmatrix}
  O
  & 
  0
  \\
  0 & -\omega_n 
  \end{bmatrix}
  +
  \frac{\xi_n}{\omega_n}
  \begin{bmatrix}
  E_{n-1}+\omega^\prime{}^t\omega^\prime/\omega_n^2 & 0
  \\
  -{}^t\omega^\prime/\omega_n & 0 
  \end{bmatrix}
  \right\}
\\
  =
& 2\pi
  \begin{bmatrix}
  \dfrac{\xi_n}{\omega_n}
  \left(E_{n-1}+\dfrac{\omega^\prime{}^t\omega^\prime}{\omega_n^2}\right)
  & 0
  \\
  -\dfrac{{}^t\omega^\prime}{\omega_n} 
  & 
  -i\omega_n
  \end{bmatrix} 
\end{align*}
and 
\begin{align*}
  \det\phi^{\prime\prime}_{z\zeta}(z,\omega,t)
& =
  \det\bigl(L \phi^{\prime\prime}_{z\zeta}(z,\omega,t) QR\bigr)
\\
& =
  (2\pi)^n
  \left(\frac{\xi_n}{\omega_n}\right)^{(n-1)}
  (-i\omega_n)
  \det\left(E_{n-1}+\dfrac{\omega^\prime{}^t\omega^\prime}{\omega_n^2}\right).  
\end{align*}
In Step~2 we have already proved that 
$E_{n-1}+\omega^\prime{}^t\omega^\prime/\omega_n^2>0$, 
and then we obtain $\det\phi^{\prime\prime}_{z\zeta}(z,\omega,t)\ne0$ 
at all the non-degenerate critical points. 
\end{proof}
\begin{proof}[{\bf Step~4 of Proof of Theorem~\ref{theorem:bh}}]
Fix arbitrary $(x_0,\xi_0) \in T^\ast(\mathbb{R}^n)\setminus0$. 
Let $\rho$ be a small positive number satisfying 
$\rho\leqq\min\{\lvert\xi_0\rvert/2,1\}$. 
We denote by $B_\rho(a)$ the open $n$-ball of radius $\rho$ and center $a \in \mathbb{R}^n$. 
We observe the behavior of $\mathcal{B}_hU(x-i\xi)$ for $(x,\xi) \in B_\rho(x_0){\times}B_\rho(\xi_0)$. In particular we evaluate the influence of the degenerate critical points of the phase function. Note that $0 \not\in B_\rho(\xi_0)$ since $\xi_0\ne0$ and $0<\rho\leqq\lvert\xi_0\rvert/2$. 
\par
We introduce subsets of $\mathbb{S}^{n-1}$ and bounded intervals in $\mathbb{R}$ to define a cut-off function of $(\omega,t)\in\mathbb{P}^n$. 
Set 
\begin{align*}
  S_\rho(\xi_0)
& :=
  \left\{
  \frac{\xi}{\lvert\xi\rvert}
  \ : \ \xi \in B_\rho(\xi_0)
  \right\},
\\
  S_\rho^\perp(\xi_0)
& :=
  \left\{
  \omega\in\mathbb{S}^{n-1}
  \ : \ \nu\omega=0 \ \text{with some}\ \nu \in S_\rho(\xi_0)
  \right\}.
\end{align*}
We have $\lvert\nu\omega\rvert<1/4$ 
for any $\nu \in S_\rho(\xi_0)$ and for any $\omega \in S_\rho^\perp(\xi_0)$ 
provided that $\rho$ is sufficiently small. 
If we set 
$$
C_\rho(\xi)
:=
\{
\omega\in\mathbb{S}^{n-1}
\ : \ 
\lvert\nu\omega\rvert<1/2 
\ \text{for any}\ \nu \in S_\rho(\xi_0)
\}, 
$$
then $\overline{S_\rho^\perp(\xi_0)} \subset C_\rho(\xi_0)$ in $\mathbb{S}^{n-1}$. 
Finally we set 
\begin{align*}
  T_0
& =
  \sup
  \{
  \lvert{x}\omega\rvert
  \ : \ 
  x \in B_\rho(x_0), 
  \omega \in S_\rho^\perp(\xi_0)
  \},
\\
  T_1
& =
  \sup
  \{
  \lvert{x}\omega\rvert
  \ : \ 
  x \in B_\rho(x_0), 
  \omega \in C_\rho(\xi_0)
  \}. 
\end{align*}
Then $0<T_0<T_1$. 
We here remark that all the degenerate critical points $(\omega,t)$ 
for $(x,\xi) \in B_\rho(x_0){\times}B_\rho(\xi_0)$, 
which are given by  $x\omega=t$ and $\xi\omega~$, are contained in 
$S_\rho^\perp(\xi_0)\times[-T_0,T_0]$. 
If $(t,\omega) \in S_\rho^\perp(\xi_0)\times[-T_0,T_0]$, 
then $(-t,-\omega) \in S_\rho^\perp(\xi_0)\times[-T_0,T_0]$ 
and therefore $S_\rho^\perp(\xi_0)\times[-T_0,T_0]$ 
can be regarded as a subset of $\mathbb{P}^n$. 
\par
We pick up 
$\chi_1(\omega) \in C^\infty(\mathbb{S}^{n-1})$ 
and
$\chi_2(t) \in C^\infty(\mathbb{R})$ 
such that 
$$
0 \leqq \chi_1(\omega), \chi_2(t) \leqq 1, 
\quad
\chi_1(-\omega)=\chi_1(\omega), 
\quad
\chi_2(-t)=\chi_2(t), 
$$
$$
\chi_1(\omega)
=
\begin{cases}
1, &\ \omega \in S_\rho^\perp(\xi_0),
\\
0, &\ \omega \not\in C_\rho(\xi_0), 
\end{cases}
\quad
\chi_2(t)
=
\begin{cases}
1, &\ t\in[-T_0,T_0],
\\
0 &\ t\not\in(-T_1,T_1).   
\end{cases}
$$
If we set $\chi(\omega,t):=\chi_1(\omega)\chi_2(t)$, 
then $\chi(\omega,t)$ is a smooth function on $\mathbb{P}^n$, 
$\operatorname{supp}[\chi]$ is contained in 
$\overline{C_\rho(\xi_0)}\times[-T_1,T_1]$, 
and $\operatorname{supp}[1-\chi]$ is contained in 
$\mathbb{P}^n\setminus\bigl(S_\rho^\perp(\xi_0)\times(-T_0,T_0)\bigr)$. 
Let $U \in \mathscr{S}^\prime(\widetilde{\mathbb{P}^n})$. 
We split $\mathcal{B}_hU$ into two parts: 
$$
\mathcal{B}_hU=\mathcal{B}\bigl((1-\chi)U\bigr)+\mathcal{B}(\chi{U}).
$$
The first term $\mathcal{B}\bigl((1-\chi)U\bigr)(x-i\xi)$ 
of the right hand side of the above 
is independent of the degenerate critical points for 
$(x,\xi) \in B_\rho(x_0){\times}B_\rho(\xi_0)$. 
So we evaluate $e^{-\pi\xi^2/h}\mathcal{B}(\chi{U})(x-i\xi)$ 
for $(x,\xi) \in B_\rho(x_0){\times}B_\rho(\xi_0)$. 
\par
We now recall 
\begin{align*}
  \mathcal{B}_h(\chi{U})(x-i\xi) 
& =
  A_nh^{-3n/4}
  \iint_{C_\rho(\xi_0)\times[-T_1,T_1]}
  U(\omega,t)\Phi(\omega,t;x,\xi)
  d\omega
  dt,
\\
  \Phi(\omega,t;x,\xi)
& =
  e^{-\pi(x\omega-t-i\xi\omega)^2/h}
  H_{n-1}
  \left(\sqrt{\frac{\pi}{h}}(x\omega-t-i\xi\omega)\right)
  \chi(\omega,t). 
\end{align*}
Rigorously speaking, this is a pairing of a distribution $U$ 
and a test function $\Phi(\cdot,\cdot;x,\xi)$ with parameters $(x,\xi)$ on $\mathbb{P}^n$. 
Here we recall that $\lvert\xi\omega\rvert\leqq\lvert\xi\rvert/2$ 
for $\xi \in B_\rho(\xi_0)$ and $\omega \in C_\rho(\xi_0)$. 
Then we deduce that there exist constants $C_j$ ($j=1,2,3$) 
and a positive integer $N$ such that 
for $(x,\xi) \in B_\rho(x_0){\times}B_\rho(\xi_0)$
\begin{align*}
  \lvert\mathcal{B}_h(\chi{U})(x-i\xi)\rvert
& \leqq
  C_1h^{-3n/4}
  \sum_{k=0}^N
  \sup_{(\omega,t)\in C_\rho(\xi_0)\times[-T_1,T_1]}
  \lvert\nabla_{\omega,t}^k\Phi(\cdot,\cdot;x,\xi)\rvert
\\
& \leqq
  C_2h^{-3n/4}(1+\lvert{x}\rvert+\lvert{\xi}\rvert)^{2N}
  \sup_{(\omega,t)\in C_\rho(\xi_0)\times[-T_1,T_1]}
  e^{-\pi(x\omega-t)^2/h+\pi(\xi\omega)^2/h}
\\
& \leqq
  C_3h^{-3n/4}e^{\pi\xi^2/4h}, 
\end{align*}
where $\nabla_{\omega,t}F(\omega,t)$ is the gradient of $F(\omega,t)$. 
Here we recall that $\lvert\xi\rvert\geqq\lvert\xi_0\rvert/2$ for $\xi \in B_\rho(\xi_0)$ 
since $0<\rho\leqq\lvert\xi_0\rvert/2$. 
Thus we obtain 
$$
e^{-\pi\xi^2/h}\lvert\mathcal{B}_h(\chi{U})(x-i\xi)\rvert
\leqq
C_3h^{-3n/4}e^{-3\pi\xi^2/4h}
\leqq
C_3h^{-3n/4}e^{-3\pi\xi_0^2/16h}
\leqq
C_4e^{-\pi\xi_0^2/8h}
$$
as $h \downarrow 0$ uniformly for $(x,\xi) \in B_\rho(x_0){\times}B_\rho(\xi_0)$ with some constant $C_4$. This completes the proof. 
\end{proof}
%
%
\section{Invariant expression of canonical transform}
\label{section:WF}
Finally we shall prove Theorem~\ref{theorem:kb}. 
Firstly we obtain $\kappa_B^{-1}(\Lambda_\Phi)$. 
For this purpose we find 
the invariant expression of ${}^t[\eta,\tau]=-\phi^\prime_\zeta$. 
In other words, we will look for the form 
which is independent of the choice of the local coordinates of 
$\mathbb{P}^n$. 
Secondly we will compute 
$\kappa_B(\omega,t,2\pi\tau\eta,2\pi\tau)$ 
by using $\kappa_B^{-1}(\Lambda_\Phi)$. 
\begin{proof}[{\bf Proof of Theorem~\ref{theorem:kb}}] 
Fix arbitrary $\omega\in\mathbb{S}^{n-1}$, 
and pick up an orthonormal basis $\{\nu_1,\dotsc,\nu_{n-1},\omega\}$ of $\mathbb{R}^n$. 
We introduce 
$$
\mu
=
\mu(s_1,\dotsc,s_{n-1}) 
:=
s_1\nu_1+\dotsb+s_{n-1}\nu_{n-1}+\sqrt{1-s_1^2-\dotsb-s_{n-1}^2}\omega
\in \mathbb{S}^{n-1}
$$
with $s_1^2+\dotsb+s_{n-1}^2<1/2$. 
We consider $\bigl(\mu(s_1,\dotsc,s_{n-1}),t\bigr) \in \mathbb{P}^n$ 
near $(\omega,t)$ with local coordinates 
$(s^\prime,t)=(s_1,\dotsc,s_{n-1},t)$. 
In the same way as Part~1 of Proof of Theorem~\ref{theorem:bh}, 
we can deduce that the critical point of the form 
$(\omega,t)=\bigl(\mu(0),t\bigr)$ must satisfy 
$$
x\omega=t,
\quad
\xi\ne0,
\quad
(\xi\omega)^\prime_{s^\prime}
:=
\frac{\partial}{\partial s^\prime}(\xi\mu(s^\prime))\Bigg\vert_{s^\prime=0}
=0.
$$
We compute the last condition of the above. 
Since $\sqrt{1-\sigma}=1+\mathcal{O}(\sigma)$ near $\sigma=0$, 
we deduce that 
\begin{equation}
\mu(s^\prime)
=
\omega
+
s_1\nu_1+\dotsb+s_{n-1}\nu_{n-1}
+
\mathcal{O}\bigl((s^\prime)^2\bigr)
\label{equation:jurong}
\end{equation}
near $s^\prime=0$. Then 
$$
\xi\mu(s^\prime)
=
\xi\omega+s_1\xi\nu_1+\dotsb+s_{n-1}\xi\nu_{n-1}
+
\mathcal{O}\bigl((s^\prime)^2\bigr)
$$
and $(\xi\omega)^\prime_{s^\prime}=0$ implies that 
$\xi\nu_1=\dotsb=\xi\nu_{n-1}=0$. 
Hence we obtain $\xi=(\xi\omega)\omega$ 
and $\omega=\pm\xi/\lvert\xi\rvert$. 
Next we compute $-\phi^\prime_\zeta$ at 
$\zeta=(s^\prime,t)$ with $s^\prime=0$. 
Using \eqref{equation:jurong} again, we deduce that 
\begin{align*}
  -\phi^\prime_\zeta\bigl(z,\mu(s^\prime),t\bigr)
  \Big\vert_{x\omega=t, \omega=\pm\xi/\lvert\xi\rvert, s^\prime=0}
& = 
  2\pi i \bigl(z\mu(s^\prime)-t\bigr)
  \begin{bmatrix}
  -\bigl(z\mu(s^\prime)\bigr)^\prime_{s^\prime}
  \\
  1 
  \end{bmatrix}
  \Bigg\vert_{x\omega=t, \omega=\pm\xi/\lvert\xi\rvert, s^\prime=0}. 
\end{align*}
Using $\xi\nu_1=\dotsb=\xi\nu_{n-1}=0$ again, we have 
$$
-
z\mu(s^\prime)
=
-
z\omega
-
s_1x\nu_1
-
\dotsb
-
s_{n-1}x\nu_{n-1}
+
\mathcal{O}\bigl((s^\prime)^2\bigr). 
$$
Hence we obtain 
$$
-\phi^\prime_\zeta\bigl(z,\mu(s^\prime),t\bigr)
\Big\vert_{x\omega=t, \omega=\pm\xi/\lvert\xi\rvert, s^\prime=0}
= 
2\pi i (-i\xi\omega)
\begin{bmatrix}
-x\nu_1
\\
\vdots
\\
-x\nu_{n-1}
\\
1 
\end{bmatrix}
\Bigg\vert_{\omega=\pm\xi/\lvert\xi\rvert}
=
\pm
2\pi\lvert\xi\rvert
\begin{bmatrix}
-x\nu_1
\\
\vdots
\\
-x\nu_{n-1}
\\
1 
\end{bmatrix}.
$$
Here we recall that 
$-{}^t\phi^\prime_\zeta\in\omega^\perp\times\mathbb{R}$. 
So we can express 
${}^t[-x\nu_1,\dotsc,-x\nu_{n-1}]\in\mathbb{R}^{n-1}$ 
as 
$$
-
\sum_{j=1}^{n-1}
(x\nu_j)\nu_j
=
-x+(x\omega)\omega
=
-x+\frac{(x\xi)\xi}{\lvert\xi\rvert^2}
\quad
\text{in}
\quad
\omega^\perp.
$$
The right hand side in the above equation is 
invariant of the choice of the local coordinates of $\mathbb{P}^n$. 
We see this as the invariant expression of $\eta$ of ${}^t[\eta,\tau]$ 
which is the fiber variables of $T^\ast(\mathbb{P}^n)$. 
Thus we have 
$$
\kappa_B(x-i\xi,2\pi\xi)
=
\pm
\left(
\frac{\xi}{\lvert\xi\rvert}, 
\frac{x\xi}{\lvert\xi\rvert}, 
2\pi\lvert\xi\rvert
\left(
-x+\frac{(x\xi)\xi}{\lvert\xi\rvert^2}
\right),
2\pi\lvert\xi\rvert
\right)
$$
for $(x,\xi) \in T^\ast(\mathbb{R}^n)\setminus0$. 
\par
Finally we consider the inverse of $\kappa_B^{-1}$. 
Suppose that 
$(\omega,t,\eta^\prime,\tau^\prime) \in T^\ast(\mathbb{P}^n)$ 
satisfies 
$$
\kappa_B(\omega,t,\eta^\prime,\tau^\prime)
=
(x-i\xi,2\pi\xi)\in\Lambda_\Phi,
\quad
(x,\xi) \in T^\ast(\mathbb{R}^n)\setminus0.   
$$
Then we have 
$$
\omega
=
\pm
\frac{\xi}{\lvert\xi\rvert}, 
\quad
t
=
\pm
\frac{x\xi}{\lvert\xi\rvert},
\quad
\eta^\prime
=
\pm
2\pi\lvert\xi\rvert
\left(
-x+\frac{(x\xi)\xi}{\lvert\xi\rvert^2}
\right),
\quad
\tau^\prime
=
\pm
2\pi\lvert\xi\rvert.
$$
This implies that 
$$
t
=
x\omega,
\quad
t\omega
=
\frac{(x\xi)\xi}{\lvert\xi\rvert^2}, 
\quad
\eta^\prime
=
\tau^\prime(-x+t\omega), 
\quad
\tau^\prime\omega
=
2\pi\xi.
$$
Here we use new variables 
$(\omega,t,2\pi\tau\eta,2\pi\eta) \in \mathbb{P}^n$, 
that is, $\tau^\prime=2\pi\tau$ and $\eta^\prime=2\pi\tau\eta$. 
Then we have 
$$
t
=
x\omega,
\quad
t\omega
=
\frac{(x\xi)\xi}{\lvert\xi\rvert^2}, 
\quad
\eta
=
-x+t\omega, 
\quad
\tau\omega
=
\xi. 
$$
This is valid for $\tau\ne0$ since $\xi\ne0$. 
Hence we have 
$$
\kappa_B(\omega,t,2\pi\tau\eta,2\pi\tau)
=
(t\omega-\eta-i\tau\omega,2\pi\tau\omega)
$$
for $\tau\ne0$. 
Note that this is valid also for the case of $\tau=0$.  
This completes the proof. 
\end{proof}
%
%
\begin{center}
{\sc Acknowledgment} 
\end{center}
\par
The author would like to express his sincere gratitude to the referees who carefully read the manuscript and provided many valuable comments. 
%
%

\end{document}